\pdfoutput=1

\documentclass{article}

\usepackage{mathtools}

\usepackage{graphicx}
\usepackage{amsfonts}
\usepackage{amssymb}
\usepackage{amsmath}
\usepackage{amsthm}
\usepackage{url}
\usepackage{color}
\usepackage{multirow}
\usepackage{enumerate}
\usepackage{hyperref}
\usepackage{epstopdf}
\usepackage{tikz}
\usepackage{here}
\usepackage{bigstrut}
\usepackage{pgfplots}
\usepackage{subcaption}
\usepackage{algorithm}
\usepackage{algpseudocode}

\DeclareMathOperator{\trace}{trace}
\DeclareMathOperator{\Diag}{Diag}
\DeclareMathOperator{\diag}{diag}

\newtheorem{thm}{Theorem}
\newtheorem{lem}[thm]{Lemma}
\newtheorem{cor}[thm]{Corollary}

\newtheorem{bem}[thm]{Remark}

\newcommand{\epr}{\hfill $\Box$}
\newcommand{\N}{\mathbb{N}}

\newcommand{\Proof}{{\em Proof}.~}

\begin{document}

\title{Stable-Set and Coloring bounds based on 0-1 quadratic
  optimization 
\footnote{This project is supported by the Austrian Science Fund (FWF): DOC 78.}
}
\author{
  {Dunja Pucher} \thanks{Department of Mathematics,
  University of Klagenfurt, Austria, {\tt
    dunja.pucher@aau.at}}
\and  {Franz Rendl}\thanks{Department of Mathematics,
  University of Klagenfurt, Austria,
  {\tt franz.rendl@aau.at} }}

\date{\today}

\maketitle
\begin{abstract}
We consider semidefinite relaxations of Stable-Set and Coloring, which
are based on quadratic 0-1 optimization.
Information about the stability number and the chromatic number
is hidden in the objective function. This leads to simplified
relaxations which depend mostly on the number of vertices of the
graph. We also propose tightenings of the relaxations which
are based on the maximal cliques of the underlying graph.
Computational results on graphs from the literature show
the strong potential of this new approach. 
\end{abstract}


\section{Introduction}

Let $G$ be a simple graph with vertex set $V(G)= \{1, \ldots, n \}$
and edge set $E(G)$. We also write $V$ and $E$ for short. 
We denote by $\overline{G}$ the complement graph of $G$. 
A subset $S \subseteq V$ is called stable if no edge joins vertices
in $S$. The stability number $\alpha(G)$ denotes the cardinality of a
maximum stable set in $G$.
The Stable-Set problem asks to determine $\alpha(G)$.  

A partition of $V(G)$ into $k$ stable sets $S_{1}, \ldots, S_{k}$
provides a $k$-coloring of $G$ by assigning the 'color' $i$ to all
vertices in $S_{i}$ for $i=1, \ldots, k$. This is indeed a coloring,
because adjacent vertices receive distinct colors, as they belong to
distinct stable sets. 

The smallest number $k$ such that $G$ has a $k$-partition into stable
sets is denoted by $\chi(G)$, the chromatic number of $G$. The
Coloring problem consists in determining $\chi(G)$. 
Stable-Set and Coloring are both contained in Karp's original list
of NP-complete problems from 1972, see \cite{Karp72}. 
H{\aa}stad~\cite{Hastad} shows that even getting a rough estimate on $\alpha(G)$
is NP-hard. Khanna, Linial and Safra~\cite{Khanna} show that
knowing $\chi(G)$ is of little use in actually finding a
coloring with $\chi(G)$ colors. They show that coloring
a $3$-colorable graph in polynomial time with at most $4$ colors is
not possible unless $P=NP$.

Lov$\acute{{\rm a}}$sz introduced the
following graph parameter $\vartheta(G)$. It can be
defined as the optimal value of the following 
semidefinite program
$$
\vartheta(G) = \max \sum_{i,j}X_{i,j} \mbox{ such that }\trace(X)=1, ~
X \succeq 0, ~X_{i,j}= 0 ~\forall [i,j] \in E(G).
$$
This parameter can be computed to fixed precision in polynomial time
and separates $\alpha(G)$ and $\chi(\overline{G})$
$$
\alpha(G) \leq \vartheta(G) \leq \chi(\overline{G}),
$$
as shown in \cite{Lov:79}. Since the graph parameters are integers
we have the slight strengthening
$$
\alpha(G) \leq \lfloor \vartheta(G) \rfloor \leq
\vartheta(G) \leq \lceil  \vartheta(G) \rceil 
\leq \chi(\overline{G}).
$$

\bigskip
We now consider quadratic optimization problems in 0-1
variables that serve the same purpose of getting bounds
on $\alpha(G)$ and $\chi(\overline{G})$. Consider
\begin{align}
s(k) := \min \frac{1}{2} x^{T}Ax \mbox{ such that }
x \in \{0,1 \}^{n}, ~ \sum_{i} x_{i}=k. \label{s(k)}
\end{align}
Clearly, problem (\ref{s(k)}) is infeasible if $k$ is not integer. If $k \in
\N$ and $s(k)=0$ then $G$ contains a stable set of size
$k$.
The most interesting case occurs for $k \in \N$ with $s(k)>0$ because
it shows that $G$ has no stable set of size $k$, hence we have the
upper bound
$$
\alpha(G) < k
$$
for the stability number $\alpha(G)$. 
Next, we consider 
\begin{align}
	c(k) := \min \frac{1}{2}\langle X, AX \rangle \mbox{ such that }
	X \in \{0,1\}^{n \times k}, ~ \sum_{j}X_{i,j}=1 ~\forall i. \label{c(k)}
\end{align}

Feasible matrices $X$ for this problem have exactly one nonzero entry
in each row, hence they represent partitions of $V(G)$ into $k$
partition blocks. 
Suppose $X$ is feasible for problem (\ref{c(k)}). 
The entry $X_{i,r}=1$ tells us that
vertex $i$ belongs to partition block $r$.
Therefore $X_{i,r}X_{j,r}=1$ exactly if $i$ and $j$ are in the same
block $r$. This is the basis for the following simple observation.

\begin{lem}
  The number 
  of edges joining vertices in the same partition blocks defined by $X$
  is given by  $\frac{1}{2}\langle X, AX \rangle$. 
\end{lem}
\Proof
We have
$\frac{1}{2} \langle X, AX \rangle =
\sum_{[i,j] \in E(G)} \sum_{r=1}^{k}X_{i,r}X_{j,r}$ and the second term
counts the number of edges within the same partition blocks. 
\epr

\medskip

The number $c(k)$ can be used to get lower bounds on the chromatic
number. 
If we find $k \in \N$ such that 
$c(k)>0$ then $G$ has no partition into $k$
stable sets,    
providing the lower bound
$$ k < \chi(G) $$
for the chromatic number of $G$. 
These bounds are not directly useful because both $c(k)$ and $s(k)$
are NP-hard graph parameters.

\bigskip

It is the main purpose of the present paper to introduce 
semidefinite relaxations for
$c(k)$ and $s(k)$, and investigate when their optimal value is
positive.
We 
denote these relaxations by $P(t)$ and $Q(t)$

\begin{align}
	P(t) ~~~ \min \frac{1}{2} \langle A, Y \rangle \colon
	\begin{pmatrix} Y & e \\ e^T & t \end{pmatrix} \succeq 0,
	\hspace{0.5em} \diag(Y) = e, \hspace{0.5em} Y \geq 0 \label{CC(k)},
\end{align}

\begin{align}
	Q(t) ~~ \min \frac{1}{2} \langle A, X \rangle \colon X\succeq 0,
	\hspace{0.5em} X \geq 0, \hspace{0.5em} \trace(X) = t,
	\hspace{0.5em} X e = t \diag(X) \label{SS(k)}.
\end{align}
Their optimal values are $C(t)$ and $S(t)$, respectively.

It will turn out that these relaxations are closely related to
Schrijver's refinement of $\vartheta(G)$ in case of Stable-Set and to the
Szegedy-bound in case of Coloring. We also explore
refinements for these bounds which are based
on a list of all maximal cliques of the graph in question.
Our preliminary computations indicate the potential of this new
approach.

\medskip
We close this section with some words on notation used throughout.
The vector of all-ones is denoted by $e$ and we write $J=ee^{T}$ for
the matrix of all-ones.
Let $0_n$ be the zero matrix and $I_n$ the identity matrix of order
$n$.
We denote by $e_i$ the column $i$ of the matrix $I_n$. Furthermore, we set
\begin{align*}
	E_i &\coloneqq e_i e_i^T \\
	E_{i,j} &\coloneqq (e_i + e_j)(e_i + e_j)^T.
\end{align*}

\section{Semidefinite relaxations for $c(k)$}

Computing the parameter $c(k)$ leads to an NP-hard problem, so we
are interested in tractable relaxations.
As the objective function in $c(k)$ is quadratic, it seems plausible
to explore semidefinite relaxations.  
First, we note that $\langle X, AX \rangle = \langle A, XX^T \rangle$,
and the main diagonal of the matrix $XX^T$ clearly
equals the all-ones vector $e$ for all $X$ feasible for the problem
defining $c(k)$. 
Let us extend $X$
with an additional row of all-ones to get 
$ \hat{X} := \begin{pmatrix} X \\ e^T \end{pmatrix} $. 
Then
\begin{align*}
	\hat{X}\hat{X}^T := \begin{pmatrix} XX^T & e \\ e^t & k \end{pmatrix}.
\end{align*}
Finally, we allow arbitrary symmetric matrices $Y$ instead of $XX^{T}$
and for $t \geq 1$ obtain the semidefinite relaxation $P(t)$ 
given in (\ref{CC(k)}).

Clearly, for all integers $k \in \N$ we have
$$c(k) \geq C(k) \geq 0.$$ 
We are now investigating the function $C(t)$ in more detail. 
\begin{lem}
  The function $C(t)$ is monotonically decreasing for $t \geq 1$.
\end{lem}
\begin{proof}
  Consider $t' > t \geq 1$ and suppose that $Y$ is optimal for problem
  $P(t)$. Then $Y$ is also feasible for problem $P(t')$ and therefore
  $C(t') \leq C(t)$. 
\end{proof}
This monotonicity property is the basis for the following lower bound
on $\chi(G)$.
\begin{thm}
  Let $t$ be given such that $C(t)>0$. Then $\chi(G)\geq
  \lfloor t\rfloor +1$.
\end{thm}
\begin{proof}
  Suppose $G$ has a coloring with $k= \lfloor t \rfloor$ colors given
  by the $n \times k$ partition matrix $X$. Then $XX^{T}$ is feasible
  for the problem $P(k)$ with $C(k)=0$. But then monotonicity of $C$
  shows that $C(t)=0$ since $t \geq k$, a contradiction.
\end{proof}

\begin{lem} 
  $C(1) = |E(G)|$.
\end{lem}
\begin{proof}
  It is a simple exercise to verify that the only feasible solution
  for $P(1)$ is $Y=J$. Therefore  
  $\frac{1}{2}\langle A, J \rangle = |E(G)|$.  
\end{proof}  
\begin{lem}
  $C(n)=0$.
\end{lem}
\begin{proof}
  Let
  $Y=\begin{pmatrix} I & e \\ e^T & n \end{pmatrix}$. Note that $Y$
  is singular because the sum of the first $n$ rows is equal to the
  last row. Therefore $I$ and $Y$ have the same rank (equal to $n$).
  It follows by the eigenvalue interlacing theorem
  that the $n$ nonzero eigenvalues of $Y$ are at least
  as large as the $n$ eigenvalues of $I$. Thus $Y$ is feasible for
  $P(n)$ showing that $0 \leq C(n) \leq \frac{1}{2} \langle A, I
  \rangle=0$. 
\end{proof}  

The relaxation $P(t)$ in (\ref{CC(k)}) is strictly feasible for any
$n > t > 1$
by considering for instance
$$
Y = \frac{t-1}{t}I + \frac{1}{t}J.
$$
Note in particular that $Y - \frac{1}{t}J =\frac{t-1}{t}I \succ 0$
implies that $\begin{pmatrix} Y & e \\ e^T & t \end{pmatrix} \succ 0$. 

We summarize the properties of $C(t)$ for nonempty graphs as
follows.
The function $C(t)$ is monotonically decreasing in the interval
$\left[ 1,n \right]$ with $C(1) = |E(G)|>0$ and $C(n)=0$. Therefore there
exists some value $t^{*}(G)>1$ such that $C(t)>0$ for $1 \leq t <
t^{*}(G)$ and $C(t)=0$ for $t \geq t^{*}(G)$.

We show next that $t^{*}(G)$ is actually equal to the  
lower bound for the chromatic number of $G$ given by
Szegedy~\cite{Szegedy}:
\begin{align}
	\vartheta^-(G) &= \min t \colon \begin{pmatrix} Y & e \\ e^T &
          t
        \end{pmatrix} \succeq 0, \hspace{0.5em} \textrm{diag}(Y) = e,
                                                                       \hspace{0.5em} Y_{i,j} = 0 \hspace{0.5em}
        \forall [i,j] \in E(G), \hspace{0.5em} Y \geq 0 \label{Szegedy}.
\end{align}
\begin{thm}\label{C(K)_Szegedy}
  We have $C(t)>0$ if and only if
  $t<  \vartheta^{-}(G)$.
\end{thm}
\begin{proof} We first consider problem $P(t)$ for $t = \vartheta^-(G)$.
  Let $Y$ be the optimal solution of $\vartheta^-(G)$ as stated in
  (\ref{Szegedy}). Then $Y$ is also feasible
  for problem $P(\vartheta^{-}(G))$ with
  value $0$. Therefore, $C( \vartheta^-(G)) = 0$. Now let $t' > t =
  \vartheta^-(G)$. Then according to the monotonicity of $C(t)$
  we have that $C(t') = 0$.
  In case when $t < \vartheta^-(G)$, we note that the problem~(\ref{Szegedy}) is not feasible. Hence, any $Y$ satisfying 
	\begin{align*}
		\begin{pmatrix} Y & e \\ e^T & t \end{pmatrix} \succeq 0, \hspace{0.5em} \diag(Y) = e, \hspace{0.5em} Y \geq 0
	\end{align*} 
	will have an entry $Y_{i,j} > 0$ for some $[i,j] \in E(G)$,
        and hence $C(t) > 0$.
\end{proof}
Here is an immediate consequence of this theorem. 
\begin{cor}
  The smallest $k \in \N$ with $C(k)=0$ is given by
  $k=\lceil \vartheta^{-}(G)\rceil$.
\end{cor}
   
\begin{bem}\label{relax_P(t)}
  A weaker relaxation for $c(k)$ is obtained by
  requiring in problem $P(t)$ that $Y_{i,j}\geq 0$
  only for $[i,j] \in E(G)$. Let us denote this weaker
  version by $P'(t)$ and its optimal value by $C'(t) \geq 0$.
  It is an easy exercise to verify that $C'(t)$ will
  also be monotonically decreasing until it reaches 0
  at some value $t^{*}(G)$. Since the entries $Y_{i,j}$ for
  $[i,j] \notin E(G)$ are unrestricted, it is clear that
  $C'(t)=0$ for $1 \leq t \leq \vartheta(G)$.
  Argueing as before, we also find that $C'(t)>0$ if
  $t > \vartheta(G)$. Therefore, the smallest integer $k$
  such that $C'(k)=0$ is given by $k=\lceil \vartheta(G) \rceil$.
  It is remarkable that we find these conclusions 
  without actually knowing $\vartheta(G)$ or $\vartheta^{-}(G)$. 
\end{bem}

\section{Semidefinite relaxations for $s(k)$}\label{S(k)}

Before introducing the semidefinite relaxation for $s(k)$, we note that $x^{T}Ax = \langle x, Ax \rangle = \langle A, xx^T \rangle$. Furthermore, since $x \in \{0,1 \}^{n}$ means that $x_i^2 = x_i$, the diagonal of the matrix $xx^T$ must be equal to $x$. Now, by setting $X := xx^T$ we substitute the term $xx^T$ by the symmetric matrix $X$, and relax this condition to $X \succeq xx^T$. We note that the semidefiniteness constraint $X - xx^T \succeq 0$ together with $\diag(X) = x$ implies that
\begin{align*}
	e^TXe \geq (e^Tx)^2.
\end{align*}
Since we should have equality, we add the constraints
$e^{T}x=k, ~e^TXe = k^2$. Finally, by using the Schur complement
\begin{align*}
	\begin{pmatrix} X & x \\ x^T & 1 \end{pmatrix} \succeq 0 \Leftrightarrow X - xx^t \succeq 0,
\end{align*}
we arrive to the following formulation
\begin{align}
	\min \frac{1}{2} \langle A, X \rangle \colon \begin{pmatrix} X & x \\ x^T & 1 \end{pmatrix} \succeq 0, \hspace{0.5em} \diag(X) = x, \hspace{0.5em}  e^T X e = k^2,\hspace{0.5em} e^T x = k. \label{SDP_S(k)_1}
\end{align}
Unfortunately, the SDP relaxation~(\ref{SDP_S(k)_1}) is not strictly feasible, since the matrices from the feasible set
\begin{align*}
	\mathcal{F}_1 := \Big\{(X, x) \colon \begin{pmatrix} X & x \\ x^T & 1 \end{pmatrix} \succeq 0, \hspace{0.5em} \diag(X) = x,\hspace{0.5em} e^T X e = k^2, \hspace{0.5em}  e^T x = k \Big\}
\end{align*}
are singular, as shown in the following Lemma.
\begin{lem}\label{singular}
  Let $(X, x) \in \mathcal{F}_1$.
  Then $\begin{pmatrix} X & x \\ x^T & 1 \end{pmatrix}$ is singular and $Xe = k x$.
\end{lem}
\begin{proof}
	We note that
	\begin{align*}
		\begin{pmatrix} e \\ -k \end{pmatrix} ^T \begin{pmatrix} X & x \\ x^T & 1 \end{pmatrix} \begin{pmatrix} e \\ -k \end{pmatrix} &= e^T X e - 2 k e^T x + k^2 = k^2 - 2k^2 + k^2 = 0. 
	\end{align*}
	Since $\begin{pmatrix} X & x \\ x^T & 1 \end{pmatrix} \succeq 0$, we have that $\begin{pmatrix} X & x \\ x^T & 1 \end{pmatrix} \begin{pmatrix} e \\ -k \end{pmatrix} = 0$. Thus, $Xe = k x$. 
\end{proof}
\begin{bem}
The set $\mathcal{F}_1$ has no interior.
\end{bem}
In view of the previous lemma, we consider the following set
\begin{align*}
	\mathcal{F}_2 \coloneqq \{(X, x) \colon X \succeq 0, \hspace{0.5em} \diag(X) = x,\hspace{0.5em}  e^T x = k, \hspace{0.5em} X e = k x  \}
\end{align*}
and show that $\mathcal{F}_1 = \mathcal{F}_2$.

\begin{lem}\label{L2}
	$\mathcal{F}_1 = \mathcal{F}_2$.
\end{lem}
\begin{proof}
	First, we show that $\mathcal{F}_1 \subseteq \mathcal{F}_2$. Let $(X, x) \in \mathcal{F}_1$. Then $X \succeq 0$, $\diag(X) = x$, $e^T x = k$ and $e^T X e = k^2$. Additionally, due to Lemma~\ref{singular} we have that $X e = k x$. Hence, $(X, x) \in \mathcal{F}_2$.
	
	Now we show that $\mathcal{F}_2 \subseteq \mathcal{F}_1$. Let $(X, x) \in \mathcal{F}_2$. Then $X \succeq 0$, $\diag(X) = x$, $e^T x =k$ and $X e = k x$. Thus, $e^T X e = k e^T x = k^2$. Furthermore, we show that 
	$\begin{pmatrix} X & x \\ x^T & 1 \end{pmatrix} \succeq 0$.
	
	Let $Y \coloneqq \begin{pmatrix} X & x \\ x^T & 1 \end{pmatrix}$. Since $X e = k x$, we have $Y = \begin{pmatrix} X & \frac{1}{k}Xe \\ \frac{1}{k}e^TX & 1 \end{pmatrix}$. Let $X v = 0$ and set $w = \begin{pmatrix} v \\ 0 \end{pmatrix}$. Then
	\begin{align*}
		Yw = \begin{pmatrix} X & \frac{1}{k}Xe \\ \frac{1}{k}e^T X & 1 \end{pmatrix} \begin{pmatrix} v \\ 0 \end{pmatrix} = \begin{pmatrix} Xv \\ \frac{1}{k} e^T X v \end{pmatrix} = \begin{pmatrix} 0 \\ 0 \end{pmatrix}. 
	\end{align*}
	Therefore, the nullspace of $X$ (extended with zero) is contained in the nullspace of $Y$. Furthermore,
	\begin{align*}
		Y \begin{pmatrix} e \\ -k \end{pmatrix} = \begin{pmatrix} X & \frac{1}{k}Xe \\ \frac{1}{k}e^T X & 1 \end{pmatrix} \begin{pmatrix} e \\ -k \end{pmatrix} = \begin{pmatrix} Xe - Xe \\ \frac{1}{k}e^T X e - k \end{pmatrix} = \begin{pmatrix} 0 \\ 0 \end{pmatrix}.
	\end{align*}
	So $\begin{pmatrix} e \\ -k \end{pmatrix}$ is in the nullspace of $Y$ and therefore rank$(X)$ = rank$(Y)$. If $X \succeq 0$, then all nonzero eigenvalues of $X$ are positive. Hence, by the interlacing property between the eigenvalues of $X$ and $Y$, we have that all nonzero eigenvalues of $Y$ are positive. Thus, $Y \succeq 0$ and $\mathcal{F}_2 \subseteq \mathcal{F}_1$.
\end{proof}

Altogether, for $t \geq 1$ we obtain the semidefinite relaxation
$Q(t)$ 
given in~(\ref{SS(k)}).
Its optimal value $S(k)$ satisfies
$$
s(k) \geq S(k) \geq 0 \mbox{ for }k \in \N.
$$
The problem $Q(t)$ is strictly feasible for $t < n$. To show this, we consider the matrix $X$ with diagonal entries equal to $\frac{t}{n}$ and all other entries equal to $\frac{t(t-1)}{n(n-1)}$. Then $X \geq 0$, $\trace(X) = t$, $X e = t \diag(X)$, and since $t < n$, we have that $X \succ 0$.

\begin{lem}
	 $S(n) = |E(G)|$. 
\end{lem}
\begin{proof}
	For $t = n$ all elements in the matrix $X$
	will be equal to one, so the optimal value of the relaxation $S(n)$
	will be $\frac{1}{2} \langle A, J \rangle = |E(G)|$.
\end{proof}

\begin{lem}
	$S(1) = 0$.
\end{lem}
\begin{proof}
	For $t = 1$ we have that $Xe = \diag(X)$. Since $X \geq 0$, all off-diagonal entries in the matrix $X$ are zero, and therefore $S(1) = \frac{1}{2}\langle A, X \rangle = 0$.
\end{proof}

Next we show monotonicity of $S(t)$.
The monotonicity of $C(t)$ was an easy consequence of the fact that
feasible matrices for $P(t')$ are directly available from feasible
matrices for $P(t)$ for any $t > t' \geq 1$. To show monotonicity of
$S(t)$, we need a similar construction to produce feasible solutions
$X'$ for $Q(t')$ from solutions $X$ for $Q(t)$ with $t'<t$.

\begin{lem}\label{feasible}
  Let $1<t'<t$ be given and suppose $X \in Q(t)$ with
  $x=diag(X)$.
  Set $\alpha=\frac{t'(t'-1)}{t(t-1)}, ~
  \beta= \frac{t'(t-t')}{t(t-1)}$. Then
  $X':=\alpha X+ \beta Diag(x) \in Q(t')$.
\end{lem}
\begin{proof}
  Let $X \in Q(t)$ with $x=diag(X)$.
  Set $X':= \alpha X + \beta Diag(x)$. We have to select $\alpha>0$
  and $\beta>0$ such that $X' \in Q(t')$ meaning that
  $trace(X')=t', X'e=t'x'$ and $x'=\diag(X')$. The first condition
  provides the following linear equation in $\alpha$ and $\beta$:
  $$t' = \alpha t + \beta t. $$
  Next note that $X'e=(\alpha t+ \beta)x$ and $x'=(\alpha+\beta)x$
  so that the second condition becomes
  $$
  X'e = (\alpha t + \beta)x = \frac{\alpha t + \beta}{\alpha +
    \beta}x' = t'x'.
  $$
  This results in a second linear equation in $\alpha$ and $\beta$
  $$\alpha t + \beta = t'(\alpha + \beta).$$
  The unique solution is given as stated in the lemma. 
\end{proof}

\begin{lem}\label{monotonic}
	The function $S(t)$ is monotonically increasing for $t \geq
        1$. Moreover, if $S(t')>0$, then $S$ is strictly monotonically
        increasing for all $t>t'$. 
\end{lem}
\begin{proof}
Let $X' \in Q(t')$ be derived from $X \in Q(t)$ as in the previous
lemma. Then $\langle A, \Diag(x) \rangle = 0$ and we get
	\begin{align*}
 \frac{1}{2} \langle A, X' \rangle =
          \frac{1}{2}\frac{t'(t'-1)}{t(t-1)}\langle A, X \rangle
          \leq \frac{1}{2}\frac{t(t-1)}{t(t-1)} \langle A, X \rangle
          = \frac{1}{2} \langle A, X \rangle. 
	\end{align*}
        Any $X \in Q(t)$ generates a feasible matrix $X' \in Q(t')$
        so that $S(t') \leq S(t)$. 
	Finally $S(t')>0$ implies that the last inequality is strict.
\end{proof}

%
We summarize the main properties of $S(t)$ for $t \geq 1$ as follows.
$S(1)=0$ and $S(n) >0$ for nonempty graphs. Moreover, $S(t)$ is
monotonically increasing. Therefore, as in the case for $C(t)$, there
exists some value $t^{*}(G)>1$ such that $S(t)= 0$ for $t\leq t^{*}$
and $S(t)>0$ for $t>t^{*}(G)$. We show next, that this value is equal
to Schrijver's refinement of $\vartheta(G)$ towards the stability
number of $G$, given as follows, see
Schrijver~\cite{Schrijver}.
\begin{align}
	\vartheta^+ &= \max \textrm{tr}(X) \colon X - xx^t \succeq 0, \hspace{0.5em} \textrm{diag}(X) = x, \hspace{0.5em} X_{i,j} = 0 \hspace{0.5em} \forall [i,j] \in E(G), \hspace{0.5em} X \geq 0 \label{Schr1}, \\
	&= \max \langle J, X \rangle \colon X \succeq 0, \hspace{0.5em} \textrm{tr}(X) = 1, \hspace{0.5em} X_{i,j} = 0 \hspace{0.5em} \forall [i,j] \in E(G), \hspace{0.5em} X \geq 0. \label{Schr2}
\end{align}

The fact that (\ref{Schr1}) and (\ref{Schr2}) are equivalent can be found in~\cite{LovSch}.

\begin{lem}\label{link_1}
	Let $X^*$ be optimal for~(\ref{Schr1}). Then $X^*e = \vartheta^+ \diag(X^*)$. 
\end{lem}
\begin{proof}
	We set $Y^* \coloneqq \frac{1}{\vartheta^+}X^*$. Then $Y^*$ is feasible for (\ref{Schr2}), and therefore $\langle J, Y^* \rangle~\leq~\vartheta^+$ and $\langle J, X^* \rangle \leq (\vartheta^+)^2$. Since $X^* - \diag(X^*)\diag(X^*)^T \succeq 0$ we have $\langle J, X^* \rangle \geq (e^T \diag(X^*))^2 = (\vartheta^+)^2$, and thus altogether $\langle J, X^* \rangle =  (\vartheta^+)^2$. This implies 
	\begin{align*}
		e^T(X^* - \diag(X^*)(\diag(X^*))^T)e = 0,
	\end{align*}
	and since $X^* - \diag(X^*)\diag(X^*)^T \succeq 0$, we have $(X^* - \diag(X^*)\diag(X^*)^T)e = 0$, showing that $X^*e = \vartheta^+\diag(X^*)$ holds.
\end{proof}

With these results we are now able to link the optimal value of the SDP relaxation $S(k)$ with the Schrijver relaxation.

\begin{thm}\label{thm_S(k)}
	Let $A$ be the adjacency matrix of a graph $G$. Then 
	\begin{align*}
		S(t) > 0 \textrm{ if and only if } t > \vartheta^+(G).
	\end{align*}
\end{thm}
\begin{proof}
	According to the definition of $S(t)$ and setting $t = \vartheta^+(G)$, we have
	\begin{align*}
		S(\vartheta^+(G)) &= \min \frac{1}{2} \langle A, X \rangle \colon X\succeq 0, \hspace{0.5em} X \geq 0, \hspace{0.5em} \trace(X) = \vartheta^+(G), \hspace{0.5em} X e = \vartheta^+(G)\diag(X).
	\end{align*}
	Let $X$ be optimal for (\ref{Schr1}). Then $\trace(X) = \vartheta^+(G)$, $X - \diag(X)\diag(X)^T \succeq 0$, $X_{i,j} = 0$ for all $[i,j] \in E(G)$ and $X_{i,j} \geq 0$ for all $[i,j] \notin E(G)$. According to Lemma~\ref{link_1} we also have $Xe = \vartheta^+(G)\diag(X)$. Thus, $X$ is feasible for (\ref{S(k)}), and since $X_{i,j} = 0$ on $E(G)$, we conclude that $S(\vartheta^+(G)))= 0$.
	
	Now assume that $t' < t = \vartheta^+(G)$. Then according to Lemma~\ref{monotonic} we have that $S(t') = 0$.
	
	Furthermore, by considering the problem~(\ref{Schr1}), we note that $\vartheta^+(G)$ is the largest possible value for the trace of the matrix $X$ which satisfies $X \succeq 0$, $X \geq 0$, $X_{i,j} = 0$ for all $[i,j] \in E(G)$, and therefore $S(t) > 0$ for all $t > \vartheta^+(G)$.
\end{proof}

\begin{cor}
	The largest $k \in \mathbb{N}$ with $S(k) = 0$ is given by $k = \lfloor \vartheta^+(G) \rfloor$.
\end{cor}
Finally, we give a statement regarding the relationship between $S(k)$ and the stability number of a graph.
\begin{thm}\label{L6}
	Let $A$ be the adjacency matrix of a graph $G$. If $S(t)> 0$, then $\alpha(G) \leq \lfloor t \rfloor$.
\end{thm}
\begin{proof}
	We set $k = \lfloor t \rfloor + 1$. If there is a stable set of size $k$, then there exist $x \in \{0, 1\}^n$ such that $e^T x = k$ and $x^T A x = 0$. Now set $X := xx^T$ and $\diag(X) = x$. Then $\langle A, X \rangle = 0$, $X \geq 0$, $X \succeq 0$ and $\trace(X) = k$. Furthermore
	\begin{align*}
		X e = x x^T e = kx = k \diag(X).
	\end{align*}
	Thus, $S(k) = 0$, and according to Lemma~\ref{monotonic} we have $S(t) = 0$, which contradicts our assumption.
\end{proof}

\begin{bem}\label{Remark_Lovasz_S}
	A weaker relaxation for $s(k)$ is obtained by
	requiring in problem $Q(t)$ that $X_{i,j}\geq 0$
	only for $[i,j] \in E(G)$. We denote this weaker
	version by $Q'(t)$ and its optimal value by $S'(t) \geq 0$.
	One can easily verify that $S'(t)$ is monotonically increasing with
	$S'(t) = 0$ for $1 \leq t \leq \vartheta(G)$, and $S'(t) > 0$ if
	$t > \vartheta(G)$. Therefore, the largest integer $k$
	such that $S'(k) = 0$ is given by $k=\lfloor \vartheta(G) \rfloor$.
\end{bem}

\section{Strengthening of $Q(t)$ and $P(t)$}\label{Strengthen_S(k)}

First, we consider strengthening of the semidefinite relaxation $Q(t)$. For this purpose, we assume that we have found the largest $k \in \mathbb{N}$ such that $S(k) = 0$. 

A standard way to strengthen a stable set relaxation is to add cutting planes such as
triangle or odd circuit inequalities, see for example
\cite{BJR:89}, \cite{Duk:07}
and \cite{Gru:03}. 
More recently it has been suggested by 
Adams, Anjos, Rendl and Wiegele in \cite{AARW:15} to
require that for certain subsets $I \subseteq V(G)$ the corresponding
submatrix $X_{I}$ should be contained in the stable set polytope
restricted to the subgraph $G_{I}$. 
To be specific we introduce the set of all stable set vectors $S(G)$, the stable set polytope $\mathrm{STAB}(G)$, as well as the squared stable set polytope $\mathrm{STAB}^2(G)$ as follows:
\begin{align*}
	S(G) &\coloneqq \{s \in \{0, 1\}^n \colon s_is_j = 0 \quad \forall [i,j] \in E(G)\}, \\
	\mathrm{STAB}(G) &\coloneqq \textrm{conv} \{s \colon s \in S(G)\}, \\
	\mathrm{STAB^2}(G) &\coloneqq \textrm{conv} \{ss^T \colon s \in S(G)\}.
\end{align*}
The approach from~\cite{AARW:15}  asks to consider
all subsets $I \subseteq V(G)$ of increasing sizes and require that
$X_{I} \subseteq \mathrm{STAB}^{2}(G_{I})$. 
A practical
implementation of this approach is described in~\cite{Gaa:18} and
\cite{Gaa:20}.

In our approach we consider subsets $I$ which are based on 
cliques. 
As the first step, we show that for a clique $C$ it follows automatically from the relaxation $Q(t)$ that $X_C \in \mathrm{STAB^2}(G_C)$. After that we focus on subgraphs which are induced by a clique $C$ and a vertex $v$ which is not adjacent to all vertices in $C$, and derive necessary and sufficient conditions for $X_{C \cup v} \in \mathrm{STAB}^2(G_{C \cup v})$. Finally, we do the same for subgraphs whose vertices can be partitioned into two cliques $C_1$ and $C_2$, such that $C_1 \cap C_2 = \emptyset$, and such that $C_1 \cup C_2$ do not form a new clique $C_3$.

\begin{lem}\label{L3} 
	Let $C  = \{1, \dots, k\}$ be a clique. Let
	\begin{align*}
		X_C = \begin{pmatrix} x_1 & 0 & \dots & 0 \\ 0 & x_2 & \dots & 0 \\ \vdots & \vdots & \ddots & \vdots \\ 0 & 0 & \dots & x_k
		\end{pmatrix} \succeq 0.
	\end{align*}
	Then $X_C \in \mathrm{STAB}^2(G_C)$.
\end{lem}
\begin{proof}
  We use the fact that $(X,x) \in F_{1}$. This implies that 
\begin{align*}
	 \begin{pmatrix} x_1 & \dots & 0 & x_{1} \\ \vdots & \ddots & \vdots & \vdots \\ 0 & \dots & x_k & x_{k}\\ x_{1} & \dots & x_{k} &  1 
\end{pmatrix} \succeq 0.
\end{align*}
Pre- and postmultiplying with $(1, \ldots, 1, -1)$ shows that
$\sum_{i = 1}^k x_{i} \leq 1$. Since the only possible stable sets are singletons or empty set, the statement holds.
\end{proof}

Now we consider subgraphs which contain only a clique $C$ and a vertex $v$ which is not adjacent to all vertices in $C$.

\begin{lem}\label{L4}
	Let $I  = \{1, \dots, k+1\}$ such that vertices \{1, \dots, k\} form a clique $C$, and vertex $k+1$ is not adjacent to all vertices in $C$. Let
	\begin{align*}
		X_I = \begin{pmatrix} x_1 & \dots & 0 & X_{1,k+1} \\ \vdots & \ddots & \vdots & \vdots \\ 0 & \dots & x_k & X_{k,k+1}\\ X_{1,k+1} & \dots & X_{k,k+1} &  x_{k+1} 
		\end{pmatrix} \succeq 0,
	\end{align*}
	and $X_I \geq 0$. Then $X_I \in \mathrm{STAB}^2(G_I)$ if and only if
	\begin{subequations}
		\begin{alignat}{2}
			X_{i, k + 1} &\leq x_i &&\quad 1 \leq i \leq k \label{R2}\\
			\sum_{i = 1}^k X_{i,k+1} &\leq x_{k+1} &&\label{R3}\\
			\sum_{i = 1}^{k+1} x_i &\leq 1 + \sum_{i = 1}^k X_{i,k+1}. &&\label{R4}
		\end{alignat}
	\end{subequations}
\end{lem}

\begin{proof}
	$X_I \in \mathrm{STAB}^2(G_I)$ if and only if there exist $\alpha \geq 0$, $\beta_i \geq 0$, $\gamma_i \geq 0$, $\delta \geq 0$ for $i \in \{1, \dots, k\}$, such that 
	\begin{align}
		\alpha + \sum_{i = 1}^k \beta_i + \sum_{i = 1}^k \gamma_i + \delta = 1, \label{e2}
	\end{align}
	and such that
	\begin{align}
		X_I = \alpha 0_{k+1} + \sum_{i = 1}^k \beta_i E_i + \sum_{i = 1}^k \gamma_i E_{i,k+1} + \delta E_{k+1} \label{X2}.
	\end{align}
There are $2k + 1$ unknowns in the matrix $X_I$ on the one side, and $2k + 1$ coefficients $\beta_i$, $\gamma_i$ and $\delta$ on the other side. Therefore, by equating the coefficients in~(\ref{X2}) we get
	\begin{alignat*}{3}
		\gamma_i &= X_{i, k+1} && \quad 1 \leq i \leq k\\
		\beta_i &= x_i - X_{i, k+1}  && \quad 1 \leq i \leq k\\
		\delta &=  x_{k+1} - \sum_{i = 1}^k X_{i,k+1}&&. 
	\end{alignat*}
Since all coefficients should be nonnegative, we get exactly the constraints (\ref{R2}) and (\ref{R3}). Furthermore, from~(\ref{e2}) it follows that
	\begin{alignat*}{4}
		\alpha = 1 - \sum_{i = 1}^{k + 1} x_i + \sum_{i = 1}^{k} X_{i, k+1},
	\end{alignat*}
	and since the coefficient $\alpha$ should also be nonnegative, we get	the constraint (\ref{R4}).
\end{proof}

\begin{bem}
	Let $I$, $G_I$ and $X_I$ be as in Lemma~\ref{L4}. Then conditions (\ref{R2}) - (\ref{R4}) are necessary and sufficient for $X_I \in \mathrm{STAB}^2(G_I)$.
\end{bem}

The statement of Lemma~\ref{L4} can be extended to subgraphs whose vertices can be partitioned into two disjoint cliques. For the start, we consider the case when there are no edges between two disjoint cliques. 

\begin{lem}\label{L5}
	Let $I = \{1, \dots, k, k +1, \dots, k+\ell\}$. Let vertices $\{1, \dots, k\}$ form a clique $C_1$, and let vertices $\{k+1, \dots, k+\ell\}$ form a clique $C_2$, such that $C_1 \cap C_2 = \emptyset$ and such that there are no edges between vertices in $C_1$ and $C_2$. Let
	\begin{align*}
		X_I = \begin{pmatrix} x_1 & \dots & 0 & X_{1,k+1} & \dots &X_{1,k+\ell}\\ \vdots & \ddots & \vdots & \vdots & \ddots & \vdots \\ 0 & \dots & x_k & X_{k,k+1} & \dots &X_{k,k+\ell} \\  X_{1,k+1} & \dots & X_{k,k+1} & x_{k+1} & \dots & 0 \\ \vdots & \ddots & \vdots & \vdots & \ddots & \vdots & \\ X_{1,k+\ell} & \dots & X_{k,k+\ell} & 0 & \dots & x_{k+\ell} 
		\end{pmatrix} \succeq 0,
	\end{align*}
	and $X_I \geq 0$. Then $X_I \in \mathrm{STAB}^2(G_I)$ if and only if
	\begin{subequations}
		\begin{alignat}{2}
			\sum_{j = 1}^{\ell} X_{i, k + j} &\leq x_{i} &&\quad 1 \leq i \leq k \label{R6}\\
			\sum_{i = 1}^{k} X_{i, k + j} &\leq x_{k+j} &&\quad 1 \leq j \leq \ell  \label{R7}\\
			\sum_{i = 1}^{k + \ell} x_i &\leq 1 + \sum_{i = 1}^{k} \sum_{j = 1}^{\ell} X_{i, k + j} &&\label{R8}
		\end{alignat}
	\end{subequations}
\end{lem}

\begin{proof}
	We proceed as in the proof of Lemma~\ref{L4}. $X_I \in \mathrm{STAB}^2(G_I)$ if and only if there exist $\alpha \geq 0$, $\beta_i \geq 0$, $\beta_{k+j} \geq 0$, $\gamma_{i, k +j} \geq 0$ for $i \in \{1, \dots, k\}$, $j \in \{1, \dots, \ell\}$ such that
	\begin{align}
		\alpha + \sum_{i = 1}^{k + \ell} \beta_i + \sum_{i = 1}^k \sum_{j = 1}^{\ell} \gamma_{i, k + j} = 1 \label{alpha},
	\end{align}
	and such that
	\begin{align}
		X_I = \alpha 0_{k+\ell} + \sum_{i = 1}^{k} \beta_i E_i + \sum_{j = 1}^{\ell} \beta_{k + j}E_{k + j} + \sum_{i = 1}^{k} \sum_{j = 1}^{\ell} \gamma_{i, k + j} E_{i, k+j} \label{13}.
	\end{align}
	
	In this case we have altogether $k\ell + k + \ell$ unknowns in the matrix $X_I$ and the same number of coefficients $\beta_i$ and $\gamma_{i, k + j}$. Hence, by equating the coefficients in~(\ref{13}) we get
	\begin{alignat*}{4}
		& \gamma_{i, k +j} &&= X_{i, k+j} && \quad 1 \leq i \leq k, \quad 1 \leq j \leq \ell   \\
		& \beta_i &&= x_i - \sum_{j = 1}^{\ell} X_{i, k+j} && \quad 1 \leq i \leq k  \\
		& \beta_{k + j} &&= x_{k+j} - \sum_{i = 1}^{k} X_{i, k+j} && \quad 1 \leq j \leq \ell.
	\end{alignat*}
	Moreover, the sum of all coefficients should be one, so we get from (\ref{alpha}) that
	\begin{alignat*}{4}
		&\alpha &&= 1 - \sum_{i = 1}^{k + \ell} x_i + \sum_{i = 1}^{k} \sum_{j = 1}^{\ell} X_{i, k+j},
	\end{alignat*}
	and since all coefficients should be nonnegative, we get the constraints (\ref{R6}) - (\ref{R8}).
\end{proof}

\begin{bem}
	Let $I$, $G_I$ and $X_I$ be as in Lemma~\ref{L5}. Then conditions (\ref{R6}) - (\ref{R8}) are necessary and sufficient for $X_I \in \mathrm{STAB}^2(G_I)$.
\end{bem}

\begin{bem}
	Lemma~\ref{L4} is a special case of Lemma~\ref{L5}.
\end{bem}
	
Finally, we show that the statement of Lemma~\ref{L5} holds for all subgraphs whose vertices can be partitioned into two disjoint cliques.

\begin{lem}\label{C1}
Let $I = \{1, \dots, k, k +1, \dots, k+\ell\}$. Let vertices $\{1, \dots, k\}$ form a clique $C_1$, and let vertices $\{k+1, \dots, k+\ell\}$ form a clique $C_2$, such that $C_1 \cap C_2 = \emptyset$, and such that at least one vertex $i \in C_1$ and one vertex $j \in C_2$ are not adjacent. Then $X_I \in \mathrm{STAB}^2(G_I)$ if and only if constraints (\ref{R6})~-~(\ref{R8}) hold.
\end{lem}

\begin{proof}
	We assume without loss of generality that only vertices $1$ and $k+1$ are adjacent. Then $X_{1, k+1} = 0$. From the proof of Lemma \ref{L5} we know that $\gamma_{1, k + 1} = X_{1, k + 1} = 0$. Thus, the coefficient $\gamma_{1, k + 1}$ is nonnegative, so the statement holds.
\end{proof}

We sum up cutting planes from Lemmas~\ref{L4}~and~\ref{L5}. First, we note that the constraint (\ref{R2}) is a well-known cutting plane for Boolean quadric polytope, see \cite{For:60}. Second, we note that for two cliques $C_1$ and $C_2$ with $C_1 \cap C_2 = \emptyset$ we have that
\begin{align*}
	X_{C_1 \cup C_2} \in \mathrm{STAB^2}(G_{C_1 \cup C_2}) \quad \Leftrightarrow \quad X_I \in \mathrm{STAB^2}(G_I) \quad \forall I \subseteq C_1 \cup C_2.
\end{align*}
Therefore, if we have a clique $C_1$ and a vertex $j \in C_2$ (or vice versa), adding constraints (\ref{R6}), (\ref{R7}) and (\ref{R8}) will imply constraints (\ref{R3}) and (\ref{R4}). Thus, we bring constraints from Lemmas~\ref{L4} and~\ref{L5} together as follows: 
\begin{align}
	\sum_{i \in C_1} X_{i, j} &\leq x_{j} &&\qquad \forall C_1 \subseteq V(G), \forall j \in V(G), j \notin C_1 \label{UCC1} \\
	\sum_{i \in C_1 \cup C_2} x_i &\leq 1 + \sum_{\substack{i \in C_1 \\ j \in C_2}} X_{i, j} &&\qquad \forall C_1 \subseteq V(G),  C_2 \subseteq V(G), C_1 \cap C_2 = \emptyset \label{UCC2}.
\end{align}
We assume now that we have strengthened the relaxation $Q(t)$ with presented cutting planes (\ref{UCC1}) and (\ref{UCC2}). Then, obviously, we will add constraints for subgraphs of various sizes. However, we note that a common property of these subgraphs is that they all have the stability number $2$. The question is whether we were able to include all subgraphs with stability number $2$. This is unfortunately not the case, since there exist subgraphs $G_I$ such that $\alpha(G_I) = 2$, but $I$ does not have a structure as stated in Lemmas~\ref{L4}, ~\ref{L5} and ~\ref{C1}. One example is a subgraph which is induced by a cycle of length $5$. Nevertheless, we show that after imposing constraints (\ref{UCC1}) and (\ref{UCC2}) on all subgraphs which can be partitioned into two cliques, there is only one inequality which should be added to the relaxation in order to have $X_I \in \mathrm{STAB}^2(G_I)$, where $I \subseteq V$, and $G_I$ is a cycle of length $5$.

\begin{lem}\label{L6}
	Let $\vert I \vert = 5$, such that the
        induced subgraph $G_I$ is a cycle of length $5$. Let $X_I
        \succeq 0$ and $X_I \geq 0$. If constraints (\ref{UCC1}) and
        (\ref{UCC2}) are satisfied for all cliques $C_1 \subseteq I$
        and $C_2 \subseteq I$, such that $C_1 \cap C_2 = \emptyset$,
        then $X_I \in \mathrm{STAB}^2(G_I)$ if and only if 
	\begin{align}
		\sum_{i = 1}^{5} x_i &\leq 1 + \sum_{i,j \in E(G_I)} X_{i, j}. &&\label{P1} 
	\end{align}  
\end{lem}

\begin{proof}
	Without loss of generality, let $I = \{1, \dots, 5\}$ and let $\vert E(G_I) \vert = 5$, $[1,2], [2,3], [3,4], [4,5], [1,5] \in E(G_I)$. Then
	\begin{align*}
		X_I = \begin{pmatrix} x_1 & 0 & X_{1, 3} & X_{1,4} & 0 \\ 0 & x_2 & 0 & X_{2, 4} & X_{2, 5} \\ X_{1, 3} & 0 & x_3 & 0 & X_{3,5} \\  X_{1,4} & X_{2, 4} & 0 & x_4 & 0 \\ 0 & X_{2, 5} & X_{3, 5} & 0 & x_5 \\ 
		\end{pmatrix}.
	\end{align*}
	Analogously to previous proofs, we have that $X_I \in \mathrm{STAB}^2(G_I)$ if 
	\begin{align}
		X_I = \alpha 0_{5} + \sum_{i = 1}^5 \beta_i E_i + \sum_{i,j \in E(G_I)} \gamma_{i, j} E_{i,j},
	\end{align}   
	where the coefficients $\alpha$, $\beta_i$ for all $i \in I$ as well as $\gamma_{i,j}$ for all for $i,j \in E(G_I)$ are nonnegative, and 
	\begin{align}
		\alpha + \sum_{i = 1}^k \beta_i + \sum_{i = 1}^k \gamma_i = 1. \label{e3}
	\end{align}
	Explicitly written, this means that $X_I \in \mathrm{STAB}^2(G_I)$ iff the constraint~(\ref{P1}) as well as constraints
	\begin{subequations}
		\begin{align}
			x_1 &\geq X_{1, 3} + X_{1, 4} \label{p1}\\
			x_2 &\geq X_{2, 4} + X_{2, 5} \label{p2}\\
			x_3 &\geq X_{1, 3} + X_{3, 5} \label{p3}\\
			x_4 &\geq X_{1, 4} + X_{2, 4} \label{p4}\\
			x_5 &\geq X_{2, 5} + X_{3, 5} \label{p5},
		\end{align}
	\end{subequations}
	are satisfied.
	Since the maximal cliques in this graph are edges, we have according to our assumption that for all edges $i,j \in E(G_I)$ and all vertices $k \in I$, $k \neq i$, $k \neq j$, it holds that
	\begin{align*}
		x_k \geq X_{k, i} + X_{k, j},
	\end{align*}
	so the constraints (\ref{p1}) - (\ref{p5}) hold. Thus, adding the constraint~(\ref{P1}) results in $X_I \in \mathrm{STAB}^2(G_I)$.
\end{proof}

Now we consider strengthening of the semidefinite relaxation $P(t)$. Again, we assume that we have found the smallest $k \in \mathbb{N}$ such that $C(k) = 0$. Our approach is the same as for $Q(t)$---we exploit the structure of the convex hull of the set of all coloring matrices for subgraphs with certain properties.

For a graph $G$ with $V(G) =\{1, \dots, n\}$, let $S = (s_1, \dots , s_k)$ be a matrix where each column $s_i$ represents a stable set vector, and where the corresponding stable sets partition $V$ into $k$ sets. The $n \times n$ matrix $X = SS^T$ is called coloring matrix. The convex hull of the set of all coloring matrices of $G$ is denoted by
\begin{align*}
	\text{COL}(G) = \text{conv} \{X \colon X \text{ is a coloring matrix of } G\}.
\end{align*}

Since the projection of the coloring problem onto a subgraph shares the same structure as the original problem, we employ again the approach from~\cite{AARW:15}, and strengthen the relaxation $C(k)$ by requiring for $I \subseteq V(G)$ that
$ X_I \in \mathrm{COL}(G_I)$ to the formulation. Analogously to the strengthening of $Q(t)$, we focus on subgraphs with certain structures. More precisely, we show that for a subgraphs induced by a clique $C$ it follows automatically from the relaxation that $X_C \in \mathrm{COL}(G_C)$. Then we consider subgraphs which are induced by a clique $C$ and a vertex $v$ which is not adjacent to all vertices in $C$, and derive necessary and sufficient conditions for $X_{C \cup v} \in \mathrm{COL}(G_{C \cup v})$.

For showing the statement we define matrices $M_{i,j}$ as follows: let $I \subseteq V(G)$ with $I = \{1, \dots, k\}$, and let $i, j \in I$, $i \neq j$; then the entries $i,j$ and $j,i$ in the matrix $M_{i,j}$ are equal to $1$, and the rest of the entries are equal to $0$. 

Let $C = \{1, \dots, k\}$ be a clique. Then $(X_C)_{i, j} = 0$ for all $[i, j] \in E(G_C)$. Hence, the only coloring matrix of the respective subgraph $G_C$ is the identity matrix $I_k$. Thus, $X_C \in \mathrm{COL}(G_C)$.    

Now we consider subgraphs which are induced by a clique $C = \{1, \dots, k\}$ and a vertex $k + 1$ which is not adjacent to all vertices in $C$. In this case we have (up to) altogether $k+1$ different partitions of vertices into stable sets and thus (up to) $k+1$ different coloring matrices. We have
\begin{align*}
	S_i &= (e_1, \dots, e_i + e_{k+1}, \dots, e_k), \quad 1 \leq i \leq k\\
	S_{k+1} &= I_{k+1}.
\end{align*}

Hence, for $i \in \{1, \dots, k\}$ we have that $S_i S^T_i = I_{k+1} + M_{i,k+1}$, and for $i = k+1$ we have that $S_{k+1} S^T _{k+1}= I_{k+1}$. With this observation we are ready to show the following statement.

\begin{lem}\label{LCol}
	Let $I = C \cup \{k+1\}$ where $C = \{1, \dots, k\}$ is a clique and vertex $k+1$ is not adjacent to all vertices in $C$. Let
	\begin{align*}
		X_I = \begin{pmatrix} 1 & \dots & 0 & X_{1,k+1} \\ \vdots & \ddots & \vdots & \vdots \\ 0 & \dots & 1 & X_{k,k+1}\\ X_{1,k+1} & \dots & X_{k,k+1} & 1
		\end{pmatrix} \succeq 0,
	\end{align*}
	and $X_I \geq 0$. Then $X_I \in \mathrm{COL}(G_I)$ if and only if
	\begin{alignat}{2}
		\sum_{i = 1}^{k} X_{i,k+1} &\leq 1 &&\label{Col}.
	\end{alignat}
\end{lem}

\begin{proof}
	$X_I \in \mathrm{COL}(G_I)$ if and only if there exist $\alpha \geq 0$ and  $\beta_{i,k+1} \geq 0$ for $i \in \{1, \dots, k\}$, such that
	\begin{align}
		X_I = \alpha I_{k+1} + \sum_{i = 1}^{k}\beta_{i, k+1} (I_{k+1} + M _{i, k + 1}), \label{Colstr}
	\end{align}
	and such that
	\begin{align}
		\alpha + \sum_{i = 1}^{k}\beta_{i, k+1} = 1. \label{Colsum}
	\end{align}
	
	There are $k$ unknowns in the matrix $X_I$ as well as $k$ coefficients $\beta_{i, k + 1}$. Hence, by equating the coefficients in~(\ref{Colstr}) it follows that
	\begin{align*}
		X_{i, k+1} = \beta_{i, k + 1}, \quad 1 \leq i \leq k,
	\end{align*}
	so we can write the equation~(\ref{Colsum}) as
		\begin{align*}
			\alpha + \sum_{i = 1}^{k}X_{i, k+1} = 1.
		\end{align*}
	Since the coefficient $\alpha$ should be nonnegative, we get the constraint~(\ref{Col}).
\end{proof}

\begin{bem}
	Let $I \subseteq V(G)$ be as in Lemma~\ref{LCol}. Then the condition~(\ref{Col}) is necessary and sufficient for $X_I \in \mathrm{COL}(G_I)$.
\end{bem}

To sum up, we have found a new cutting plane for subgraphs induced by a clique $C$ and a vertex $j$ which is not adjacent to all vertices in the clique. Hence, we strengthen our relaxation by adding the constraint
\begin{align}
	\sum_{i \in C} X_{i, j} &\leq 1 &&\qquad \forall C \subseteq V(G), \forall j \in V(G), j \notin C \label{ECC}
\end{align}
to the formulation.

\section{Computational results}

We present the computational implementation of the problem $Q(t)$.
In order to perform computations efficiently, we switch to the problem $Q'(t)$, as already mentioned in Remark~\ref{Remark_Lovasz_S}.
We recall that for $t \leq \vartheta(G)$ the optimal value $S'(t) = 0$, while for $t > \vartheta(G)$ we have that $S'(t) > 0$.
Our goal is to find an upper bound on the stability number, i.e.\ to find the largest $k \in \mathbb{N}$ with $S(k) = 0$, since then $\alpha(G) \leq k$.
Here we note that the time needed for computation of $S'(k)$ strongly depends on the value of $k$.
More precisely, computations for $k \leq \lfloor \vartheta(G) \rfloor$ are faster than computations for $k \geq \lfloor \vartheta(G) \rfloor + 1$.
This observation is visualized in Figure~\ref{fig:example}.
Here we consider an instance from the second DIMACS implementation challenge~\cite{DIMACS}, \textsf{brock200\_1}, which has $200$ vertices and $5066$ edges, and for which we compute $\vartheta(G) = 27.4566$ within $29$ seconds\footnote{All computational experiments were performed on Intel(R) Core(TM) i7-10510U CPU @ 1.80GHz 2.30GHz with 16GB RAM under Windows 10 operating system. We perform computations in MATLAB and use MOSEK as solver. The program code associated with this paper is available as ancillary files at \url{https://arxiv.org/abs/2211.13147}.}.

\begin{figure}[h]  
	\centering 
	\begin{subfigure}[b]{0.3\linewidth}
		\begin{tikzpicture}[scale=0.45]
			\begin{axis}[
				xlabel={$k$},
				ylabel={$S'(k)$},
				xmin=20, xmax=35,
				ymin=0, ymax=25,
				xtick={20,20,21,22,23,24,25,26,27,28,29,30,31,32,33,34,35},
				ytick={0,5,10,15,20,25},
				legend pos=north west,
				ymajorgrids=true,
				grid style=dashed,
				]
				
				\addplot[
				color=blue,
				mark=square,
				]
				coordinates {
					(20,0)(21,0)(22,0)(23,0)(24,0)(25,0)(26,0)(27,0)(28,0.86)(29,2.91)(30,5.36)(31,8.13)(32,11.21)(33,14.58)(34,18.25)(35,22.19)
				};

			\end{axis}
		\end{tikzpicture}
		\caption{Values of $S'(k)$ for different values of $k$} \label{fig:example_a}  
	\end{subfigure}
	\begin{subfigure}[b]{0.3\linewidth}
		\begin{tikzpicture}[scale=0.45]
			\begin{axis}[
				xlabel={$k$},
				ylabel={CPU time},
				xmin=20, xmax=35,
				ymin=10, ymax=60,
				xtick={20,20,21,22,23,24,25,26,27,28,29,30,31,32,33,34,35},
				ytick={10,20,30,40,50,60},
				legend pos=north west,
				ymajorgrids=true,
				grid style=dashed,
				]
				
				\addplot[
				color=red,
				mark=square,
				]
				coordinates {
					(20,14.33)(21,15.29)(22,14.31)(23,15.38)(24,19.92)(25,17.03)(26,16.90)(27,16.94)(28,45.41)(29,44.91)(30,53.80)(31,47.20)(32,40.30)(33,39.61)(34,42.75)(35,39.92)};

			\end{axis}
		\end{tikzpicture}
		\caption{CPU times for the computation of $S'(k)$} \label{fig:example_b}  
	\end{subfigure}
	\begin{subfigure}[b]{0.3\linewidth}
		\begin{tikzpicture}[scale=0.45]
			\begin{axis}[
				xlabel={$k$},
				ylabel={CPU time},
				xmin=20, xmax=35,
				ymin=10, ymax=60,
				xtick={20,20,21,22,23,24,25,26,27,28,29,30,31,32,33,34,35},
				ytick={10,20,30,40,50,60},
				legend pos=north west,
				ymajorgrids=true,
				grid style=dashed,
				]
				
				\addplot[
				color=red,
				mark=square,
				]
				coordinates {
					(20,17.85)(21,17.86)(22,19.94)(23,16.89)(24,20.23)(25,18.19)(26,17.31)(27,22.69)(28,20.73)(29,17.34)(30,17.42)(31,17.75)(32,15.40)(33,15.08)(34,15.42)(35,15.37)};
			\end{axis}
		\end{tikzpicture}
		\caption{CPU times for the computation of $S^*(k)$} \label{fig:example_c}  
	\end{subfigure}
	
	\caption{Details regarding the computation of $S'(k)$ and $S^*(k)$ for graph \textsf{brock200\_1}} 
	\label{fig:example}
\end{figure}
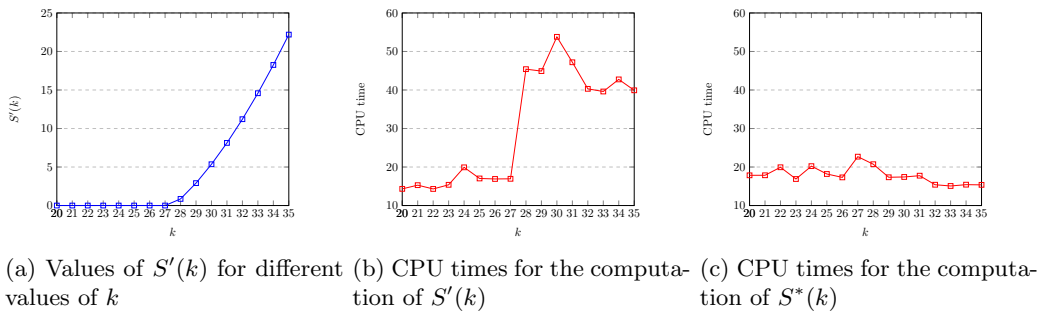  

From the Figure~\ref{fig:example_b} we see that computations of $S'(k)$ for $k \geq \lfloor \vartheta(G) \rfloor + 1$ take longer than the computation of $\vartheta(G)$. Nevertheless, there is a way to overcome this obstacle. Generally, if $S'(t) > 0$ we are not interested in the actual value of $S'(t)$. Therefore, we add the constraint $\langle A, X \rangle = 0$ into the formulation of $Q'(t)$. This newly defined problem $Q^*(t)$ will be infeasible in case when $S'(t) > 0$, which influences computational times as presented in Figure~\ref{fig:example_c}. We denote the optimal value of the problem $Q^*(t)$ by $S^*(t)$.

As the next step we investigate times needed for the computations performed on the critical interval which result with the information $\alpha(G) \leq k$.
We consider selected DIMACS instances and report results in Table~\ref{table:results_selected_DIMACS}. We state basic information about each instance---name of the graph, number of vertices and edges $n$ and $m$, stability number of the graph $\alpha(G)$, as well as information regarding $\vartheta(G)$, $k$, and CPU times needed for computations. 

\begin{table}[h!]
	\centering
	\setlength{\tabcolsep}{3pt}
	\begin{tabular}{ l r r r | r r | r r | r r } 
		\hline
		\multicolumn{4}{c}{Instance} & \multicolumn{2}{c}{\shortstack{Computation \\ of $\vartheta(G)$ }} & \multicolumn{2}{c}{\shortstack{Computation of\\ $S'(k)$ and $S'(k+1)$ }} &
		\multicolumn{2}{c}{\shortstack{Computation of\\ $S^*(k)$ and $S^*(k+1)$ }} \\
		\hline
		Graph & n & m & $\alpha(G)$ & $\vartheta(G)$ & time & $k$ & time  & $k$ & time \\ 
		\hline
		\textsf{keller4} & 171 & 5100 & 11 & 14.01 & 31 & 14 & 55 & 14 & 51\\
		\textsf{brock200\_1} & 200 & 5066  & 21 & 27.46 & 29 & 27 & 72 & 27 & 50 \\
		\textsf{brock200\_2} & 200 & 10024 & 12 & 14.23 & 159 & 14 & 295 & 14 & 195 \\ 
		\textsf{sanr200\_0.9} & 200 & 2037 & 42 & 49.27 & 5 & 49 & 11 & 49 & 13\\
		\textsf{p\_hat300\_3} & 300 & 11460 & 36 & 41.17 & 279 & 41 & 452 & 41 & 331\\
		\textsf{brock400\_1} & 400 & 20077 & 27 & 39.70 & 1598 & 39 & 2834 & 39 & 2213\\
		\hline
	\end{tabular}
	\caption{Computational times needed for the computation of $\vartheta(G)$ as well as $S'(k)$, $S'(k + 1)$, $S^*(k)$, and $S^*(k + 1)$ for $k = \lfloor \vartheta(G) \rfloor$ for selected DIMACS instances}
	\label{table:results_selected_DIMACS}
\end{table}

Finally, we examine the strengthening proposed in Section \ref{Strengthen_S(k)}.
We start with formulations $\vartheta(G)$ as well as $Q'(k)$ with $k = \lfloor \vartheta(G) \rfloor$ and proceed iteratively by adding constraints for violated inequalities only.
In each iteration we add all violations of nonnegativity constraints.
However, for constraints (\ref{UCC1}) and (\ref{UCC2}) we use a different approach.
More specifically, we consider only violations which are greater than a certain epsilon and which are greater or equal to $50\%$ of the value of the largest violation.
Altogether, iterations are done until there are less than $n$ violations.
We set $\epsilon = 0.0
1$ and perform computational experiments for selected instances from the second DIMACS implementation challenge. We have chosen instances for which the computation of all maximal cliques is done rather quickly. In order to enumerate all maximal cliques we use the Bron-Kerbosch algorithm as given in~\cite{Wildman}. Results for selected DIMACS instances are summarized in Figure~\ref{table:results_strengthening_stability_number}. We report information about considered instance---name of the graph, number of vertices and edges $n$ and $m$, stability number $\alpha(G)$, value of the $\vartheta(G)$, as well as number of maximal cliques denoted as MC. For the part concerning the strengthening of the $\vartheta(G)$ we report improved bound, number of iterations, number of added inequalities as well as CPU time, denoted as bound, iter, added and time. For the part concerning the strengthening of the problem $Q'(k)$ we report the same data, but also include information about the value of $k$ used for the first iteration.

\begin{table}[h!]
	\centering
	\setlength{\tabcolsep}{3pt}
	\begin{tabular}{ l r r r r r | r r r r | r r r r r} 
		\hline
		\multicolumn{6}{c}{Instance} & \multicolumn{4}{c}{Strengthening of $\vartheta(G)$} & \multicolumn{5}{c}{Strengthening of $Q'(k)$} \\
		\hline
		Graph & n & m & $\alpha(G)$ & $\vartheta(G)$ & MC & bound & iter & added & time & $k$ & bound & iter & added & time\\ 
		\hline
		\textsf{C125.9} & 125 & 787 & 34 & 37.80 & 532 & 35.51 & 9 & 3300 & 188 & 37 & 35 & 10 & 2768 & 127 \\
		\textsf{sanr200\_0.9} & 200 & 2037 & 42 & 49.27 & 1497 & 47.17 & 7 & 4157 & 394 & 49 & 47 & 9 & 4484 & 412 \\ 
		\textsf{brock200\_1} & 200 & 5066 & 21 & 27.45 & 11024 & 26.58 & 5 & 2519 & 1323 & 27 & 26 & 4 & 1719 & 986 \\
		\textsf{C250.9} & 250 & 3141 & 44 & 56.24 & 2470 & 54.45 & 6 & 4028 & 550 & 56 & 54 & 6 & 3622 & 365 \\ 
		\hline
	\end{tabular}
	\caption{Computational results for the strengthened upper bound on the stability number for selected DIMACS instances}
	\label{table:results_strengthening_stability_number}
\end{table}

The implementation of the problem $P(t)$ is done analogously.  First, we switch to the problem $P'(t)$ and investigate times needed for the computations of $C'(k)$ depending on the value of $k \in \mathbb{N}$. For this purpose, we consider instance \textsf{dsjc125.5} from the second DIMACS challenge. This instance has $125$ vertices and $3891$ edges, and we compute $\vartheta(G) = 11.78$ within $14$ seconds. As it can be seen from Figure~\ref{fig:example_b_C_t}, computations of $C'(k)$ for $k \leq \lceil \vartheta(G) \rceil - 1$ take longer than computations for $k \geq \lceil \vartheta(G) \rceil$. The introduction of the problem $P^*(t)$ with the additional constraint $\langle A, Y \rangle = 0$ decreases computational times for $k \leq \lceil \vartheta(G) \rceil - 1$, but increases times for other values of $k$. This is visualized in Figure~\ref{fig:example_c_C_t}.

\begin{figure}[h]  
	\centering 
	\begin{subfigure}[b]{0.3\linewidth}
		\begin{tikzpicture}[scale=0.45]
			\begin{axis}[
				xlabel={$k$},
				ylabel={$C'(k)$},
				xmin=4, xmax=19,
				ymin=0, ymax=600,
				xtick={4,5,6,7,8,9,10,11,12,13,14,15,16,17,18,19},
				ytick={0,100,200,300,400,500,600},
				legend pos=north west,
				ymajorgrids=true,
				grid style=dashed,
				]
				
				\addplot[
				color=blue,
				mark=square,
				]
				coordinates {
					(4,556.29)(5,361.41)(6,238.45)(7,156.18)(8,98.99)(9,58.22)(10,29.05)(11,9.07)(12,0)(13,0)(14,0)(15,0)(16,0)(17,0)(18,0)(19,0)
				};

			\end{axis}
		\end{tikzpicture}
		\caption{Values of $C'(k)$ for different values of $k$} \label{fig:example_a_C_t}  
	\end{subfigure}
	\begin{subfigure}[b]{0.3\linewidth}
		\begin{tikzpicture}[scale=0.45]
			\begin{axis}[
				xlabel={$k$},
				ylabel={CPU time},
				xmin=4, xmax=19,
				ymin=0, ymax=25,
				xtick={4,5,6,7,8,9,10,11,12,13,14,15,16,17,18,19},
				ytick={0,5,10,15,20,25},
				legend pos=north west,
				ymajorgrids=true,
				grid style=dashed,
				]
				
				\addplot[
				color=red,
				mark=square,
				]
				coordinates {
					(4,20.75)(5,19.82)(6,23.1)(7,21.12)(8,22)(9,20.2)(10,17.64)(11,18.54)(12,9.95)(13,8.55)(14,8.28)(15,8.22)(16,8.65)(17,10.15)(18,6.89)(19,8.64)
				};

			\end{axis}
		\end{tikzpicture}
		\caption{CPU times for the computation of $C'(k)$} \label{fig:example_b_C_t}  
	\end{subfigure}
	\begin{subfigure}[b]{0.3\linewidth}
		\begin{tikzpicture}[scale=0.45]
			\begin{axis}[
				xlabel={$k$},
				ylabel={CPU time},
				xmin=4, xmax=19,
				ymin=0, ymax=25,
				xtick={4,5,6,7,8,9,10,11,12,13,14,15,16,17,18,19},
				ytick={0,5,10,15,20,25},
				legend pos=north west,
				ymajorgrids=true,
				grid style=dashed,
				]
				
				\addplot[
				color=red,
				mark=square,
				]
				coordinates {
					(4,8.2)(5,7.76)(6,8.25)(7,7.87)(8,9.34)(9,9.08)(10,9.70)(11,10.6)(12,16.3)(13,17.09)(14,15.52)(15,15.26)(16,13.60)(17,15.51)(18,23.57)(19,15.09)
				};
			\end{axis}
		\end{tikzpicture}
		\caption{CPU times for the computation of $C^*(k)$} \label{fig:example_c_C_t}  
	\end{subfigure}
	
	\caption{Details regarding the computation of $C'(k)$ and $C^*(k)$ for graph \textsf{dsjc125.5}} 
	\label{fig:example_C_t}
\end{figure}
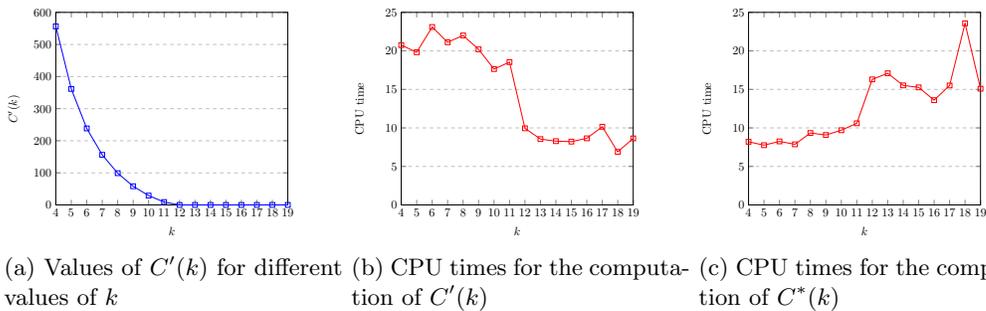  

We investigate times needed for the computations on the critical interval which result with the information that $\chi(G) \geq k + 1$. We consider DIMACS instances and present results in Table~\ref{table:results_selected_DIMACS_C_t}.

\begin{table}[h!]
	\centering
	\setlength{\tabcolsep}{3pt}
	\begin{tabular}{ l r r r | r r | r r | r r } 
		\hline
		\multicolumn{4}{c}{Instance} & \multicolumn{2}{c}{\shortstack{Computation \\ of $\vartheta(G)$ }} & \multicolumn{2}{c}{\shortstack{Computation of\\ $C'(k)$ and $S'(k+1)$ }} &
		\multicolumn{2}{c}{\shortstack{Computation of\\ $C^*(k)$ and $S^*(k+1)$ }} \\
		\hline
		Graph & n & m & $\chi(G)$ & $\vartheta(G)$ & time & $k$ & time  & $k$ & time \\ 
		\hline
		\textsf{dsjc.125.1} & 125 & 736 & 5 & 4.10 & 2 & 4 & 4 & 4 & 4 \\
		\textsf{dsjc.125.5} & 125 & 3891 & 17 & 11.78 & 13 & 11 & 29 & 11 & 28 \\
		\textsf{dsjc.250.1} & 250 & 3218 & 8 & 4.91 & 12 & 4 & 29 & 4 & 28 \\
		\textsf{dsjc.250.5} & 250 & 15668 & 28 & 16.23 & 542 & 16 & 1329 & 16 & 1681 \\
		\textsf{flat\_300\_28\_0} & 300 & 21695 & 28 & 17.00 & 1543 & 17 & 3918 & 17 &3156 \\
		\textsf{dsjc.500.1} & 500 & 12458 & 12 & 6.22 & 328 & 6 & 791 & 6 & 914 \\
		\hline
	\end{tabular}
	\caption{Computational times needed for the computation of $\vartheta(G)$ as well as $C'(k)$, $C'(k + 1)$, $C^*(k)$, and $C^*(k + 1)$ for $k = \lceil \vartheta(G) \rceil - 1$ for selected DIMACS instances}
	\label{table:results_selected_DIMACS_C_t}
\end{table}

As the last step we strengthen the relaxation $P(t)$ with constraints~(\ref{ECC}) proposed in Section~\ref{Strengthen_S(k)}. We start with the formulation $\vartheta(G)$ as well as $P'(t)$ and proceed iteratively in the same manner as for the problem $Q'(t)$, except that we set now $\epsilon = 0.05$. We consider instances from the second DIMACS challenge for which the enumeration of all maximal cliques is done in reasonable time and present results in Table~\ref{table:results_strengthening_chromatic_number}.

\begin{table}[h!]
	\centering
	\setlength{\tabcolsep}{3pt}
	\begin{tabular}{ l r r r r r | r r r r | r r r r r} 
		\hline
		\multicolumn{6}{c}{Instance} & \multicolumn{4}{c}{Strengthening of $\vartheta(G)$} & \multicolumn{5}{c}{Strengthening of $P'(k)$} \\
		\hline
		Graph & n & m & $\chi(G)$ & $\vartheta(G)$ & MC & bound & iter & added & time & $k$ & bound & iter & added & time\\ 
		\hline
		\textsf{dsjc.125.1} & 125 & 736 & 5 & 4.10 & 487 & 4.33 & 13 & 2171  & 119 & 5 & 5  & 1 & 0 & 2 \\
		\textsf{dsjc.125.5} & 125 & 3891 & 17 & 11.78 & 46494 & 13.18 & 14  & 5441 & 2065 & 12 & 14 & 12 & 4521 & 1204 \\
		\textsf{dsjc.250.1} & 250 & 3218 & 8 & 4.91 & 2584 & 5.04 & 6 & 2701 & 320 & 5 & 6 & 5 & 1801 & 134 \\
		\hline
	\end{tabular}
	\caption{Computational results for the strengthened lower bound on the chromatic number for selected DIMACS instances}
	\label{table:results_strengthening_chromatic_number}
\end{table}

From the computational results presented in Tables~\ref{table:results_strengthening_stability_number} and \ref{table:results_strengthening_chromatic_number} we can note following.
Assume that we are interested in strengthening of $\vartheta(G)$ towards the stability number of a graph, and we would like to do it in an iterative way. The first step would be to compute the value of $\vartheta(G)$ and then to add constraints iteratively, and that for violated constraints only. But assume $\vartheta(G) \notin \mathbb{N}$. Then a valid constraint for the relaxation is $e^tx \leq \lfloor \vartheta(G) \rfloor$. Therefore, the idea is to start the iterative process only after we add the constraint $e^tx \leq \lfloor \vartheta(G) \rfloor$, and to update it during the process. This procedure may be of advantage, since starting with an integral solution may yield in less violations and thus better computational times. The same approach can be used for strengthening of $\vartheta(G)$ towards the chromatic number of a graph, and that by adding the constraint $t \geq \lceil \vartheta(G) \rceil$. We consider the same instances as well as same constraints as before and present computational results in Table~\ref{table:results_strengthening_2}.

	\begin{table}
		\centering
		\setlength{\tabcolsep}{3pt}
		\raggedright
		\begin{subtable}[t]{.45\textwidth}
		\begin{tabular}{ l r r r r } 
			\hline
			\multicolumn{1}{c}{} & \multicolumn{4}{c}{Strengthening of $\lfloor \vartheta(G) \rfloor$} \\
			\hline
			Graph & bound & iter & added & time \\ 
			\hline
			\textsf{C125.9} & 35 & 7 & 2412 & 72\\
			\textsf{sanr200\_0.9} & 47 & 7 & 3588 & 237 \\
			\textsf{brock200\_1} & 26 & 4 & 1884 & 949 \\
			\textsf{C250.9} & 54 & 5 & 3579 & 313 \\
			\hline
		\end{tabular}
		\caption{Strengthening of  $\lfloor \vartheta(G) \rfloor$ towards the stability number}
		\label{table:results_strengthening_stability_number_2}
		\end{subtable}
	\raggedleft
	\begin{subtable}[t]{.45\textwidth}
		\begin{tabular}{ l r r r r } 
			\hline
			\multicolumn{1}{c}{} & \multicolumn{4}{c}{Strengthening of $\lceil \vartheta(G) \rceil$} \\
			\hline
			Graph & bound & iter & added & time \\ 
			\hline
			\textsf{dsjc.125.1} & 5 & 1 & 0 & 5 \\
			\textsf{dsjc.125.5} & 14 & 9 & 3544 & 1309 \\
			\textsf{dsjc.250.1} & 6 & 3 & 1448 & 157 \\
			\hline
		& & & &
		\end{tabular}
		\caption{Strengthening of $\lceil \vartheta(G) \rceil$ towards the chromatic number}
		\label{table:results_strengthening_chromatic_number_2}
	\end{subtable}
	\caption{Computational results for the strengthening of $\lfloor \vartheta(G) \rfloor$ and $\lceil \vartheta(G) \rceil$ for the selected DIMACS instances}
	\label{table:results_strengthening_2}
	\end{table}

 \section{Summary and Conclusion}

In this work we proposed semidefinite relaxations $P(t)$ and $Q(t)$
for two quadratic optimization problems which can be used to get
bounds on $\alpha(G)$ and $\chi(G)$. These relaxations are closely
related to the well-known Lovász Theta function $\vartheta(G)$. With
our new  relaxations we are able to give statements regarding the
bounds on $\alpha(G)$ and $\chi(G)$ without knowing the value of
$\vartheta(G)$ per se.
As shown in Figures~\ref{fig:example} and \ref{fig:example_C_t}, for
$t \geq 1$ testing whether $\alpha(G) \leq \lfloor t \rfloor$ takes
less time than the computation of $\vartheta(G)$, while testing
whether $\chi(G) \geq \lceil t \rceil$ takes less time than the
computation of $\vartheta(G)$ for certain values of $t$. Furthermore,
we proposed strengthening of the relaxations which are based on a list
of all maximal cliques of the graph in question. Here we note that the
same strengthening can be used on the standard relaxations based on
$\vartheta(G)$.
Results in Tables~\ref{table:results_strengthening_stability_number}
and \ref{table:results_strengthening_chromatic_number} show that we
were able to improve bounds and that in a reasonable time.
We have performed experiments for instances for which the enumeration
of all maximal cliques is done rather quickly.
For larger instances this might not be the case, however, we note that
the enumeration can be limited to clique sizes up to fixed $c$,
then this is still polynomial time. Finally, from
Tables~\ref{table:results_strengthening_stability_number} and
\ref{table:results_strengthening_chromatic_number} we see
that strengthening of the relaxations $P'(t)$ and $Q'(t)$ is computed
quicker than the strengthening of  $\vartheta(G)$.
This arises from the fact that is sufficient to consider relaxations
$P'(t)$ and $Q'(t)$ as well as their refinements only for $t \in
\mathbb{N}$.
But the same approach can be used for refinements of $\vartheta(G)$,
as shown in Table~\ref{table:results_strengthening_2}.

For the future it is planned to compare our approach with the standard hierarchies as proposed in~\cite{AARW:15}.

\medskip
{\bf Acknowledgement:}
We would like to thank Daniel Brosch and Elisabeth Gaar for
their useful comments.

\bibliographystyle{plain}
\bibliography{pucher_rendl}

\end{document}